\documentclass{amsart}
\usepackage{hyperref}

\usepackage[utf8]{inputenc}
\newtheorem{theorem}{Theorem}[section]

\theoremstyle{definition}

\theoremstyle{remark}
\newtheorem{remark}[theorem]{Remark}

\numberwithin{equation}{section}
\title{A generalization of Chu-Vandermonde's Identity}
\author{Seyed Saeed Naghibi}
\address{Dept. of Computer Science and Information Technology, Institute for Advanced Studies in Basic Sciences (IASBS),
Zanjan, Iran}
\email{ngsaeed@iasbs.ac.ir}

\author{Mohsen Hooshmand}
\address{Dept. of Computer Science and Information Technology, Institute for Advanced Studies in Basic Sciences (IASBS),
Zanjan, Iran}
\email{mohsen.hooshmand@iasbs.ac.ir}

\begin{document}
\begin{abstract}
We present and prove a general form of  Vandermonde's identity and use it as an alternative solution to a classic probability problem.
\end{abstract}

\maketitle

\section{Introduction}
Vandermonde's identity, its diverse proofs, and applications have been the spot of research throughout the centuries. Nowadays, it is a classic identity whose proof is available in most combinatorics-related books~\cite{knuth94}. However, there are still some efforts to give new proofs for this well-established identity, e.g.,~\cite{spivey16}. In addition to its original form proof, some specific types of its generalization have been proposed. Some works concentrated on its higher dimension~\cite{yaacov17}. Some others put their efforts into its complex~\cite{mestrovic18} or coefficients of its numbers~\cite{knuth94}.  

However, our focus was not generalizing this identity directly. We were doing some exercises in Stochastic Process and reached a classic problem during its teaching, which we put in this paper with an alternative solution in section~\ref{seq:app}. To have a formal proof of it, we reached an equation that we later realized could be considered a generalization of \textit{Vandermonde's identity}. 

This article gives a general form of Vandermonde's identity, including first order and higher orders of this identity. Then, we proof the mentioned classic problem, borrowed from Ross book~\cite{ross10} using the first order of general Vandermonde's identity. Lastly, we briefly discuss some of our thoughts about this proof.

\section{The Vandermonde's identity General form.}
At first, we encountered the first order of Vandermonde's identity. We prove it as follows.

\begin{theorem}[First Order of Vandermonde's Identity General Form.]\label{1storder} Let assume $k\leq m$ and $r-k\leq n$, then 
$\sum_{k=0}^{r}{k{m \choose k}{n \choose r-k}} = m {m+n-1 \choose r-1}
    \label{eq:mainFormula}$, where ${x \choose y}=\frac{x!}{(x-y)!y!}$, and $x,y\in  \mathbb{N}$ .
\end{theorem}

\begin{proof} 
Let start with 
\begin{equation*}
    \begin{aligned}
        m \left( x+1 \right)^{m+n-1} & = m \left( x+1 \right)^{m-1} \left( x+1 \right)^{n}\\
    \end{aligned}
\end{equation*}

By $m \left({1+x}\right) ^{m-1} = \sum_{i=0}^{m}{i{m \choose i} x^{i-1}}$~\cite{Koh} we have,

\begin{equation*}
    \begin{aligned}
       & = \sum_{i=0}^{m}{i{m \choose i} x^{i-1}} \sum_{j=0}^{n}{{n \choose j} x^{j}}.\\
    \end{aligned}
\end{equation*}
We get factor of $\frac{1}{x}$ from $x^{i-1}$ in the above equation; thus, we have
\begin{equation*}
    \begin{aligned}
       & = \frac{1}{x}\sum_{i=0}^{m}{i{m \choose i} x^{i}} \sum_{j=0}^{n}{{n \choose j} x^{j}}.\\
    \end{aligned}
\end{equation*}
Using \textit{polynomial ring}, $\sum_{i=0}^{m}{a_{i}{x^{i}}} \sum_{j=0}^{n}{b_{j}x^{j}} = 
    \sum_{r=0}^{m+n}{{ \left( \sum_{k=0}^{r}{a_{k}b_{r-k}} \right) }x^{r}}$ ~\cite{whitelow}, we convert the above formula as follows.  
\begin{equation*}
    \begin{aligned}
       & = \frac{1}{x}\sum_{r=0}^{m+n}{{ \left( \sum_{k=0}^{r}{k{m \choose k}{n \choose r-k}} \right) }x^{r}}\\
       & = \sum_{r=0}^{m+n}{{ \left( \sum_{k=0}^{r}{k{m \choose k}{n \choose r-k}} \right) }x^{r-1}}\\
        & = 0 + \sum_{r=1}^{m+n}{{ \left( \sum_{k=0}^{r}{k{m \choose k}{n \choose r-k}} \right) }x^{r-1}}
    \end{aligned}
\end{equation*}
, or
\begin{equation}
    m \left( x+1 \right)^{m+n-1} = \sum_{r=1}^{m+n}{{ \left( \sum_{k=0}^{r}{k{m \choose k}{n \choose r-k}} \right) }x^{r-1}}.
    \label{eq:firstExpansion}
\end{equation}
The equation~\ref{eq:firstExpansion} is our first expansion. On the other hand, we consider \textit{binomial expansion},
\begin{equation*}
    \left( x+1 \right)^{m+n-1} = \sum_{r=0}^{m+n-1}{{m+n-1 \choose r}}x^{r}.
\end{equation*}
By changing the index variable, the equation is
\begin{equation*}
    = \sum_{r=1}^{m+n}{{m+n-1 \choose r-1}}x^{r-1}.
\end{equation*}
Multiplying both sides with $m$ 
\begin{equation}
    m \left( x+1 \right)^{m+n-1} = \sum_{r=1}^{m+n}{m{m+n-1 \choose r-1}}x^{r-1}
    \label{eq:binomialexp}
\end{equation}
The left sides of equation~\ref{eq:firstExpansion} and equation~\ref{eq:binomialexp} are equal. Thus, 
\begin{equation*}
    \begin{aligned}
        \sum_{r=1}^{m+n}{{ \left( \sum_{k=0}^{r}{k{m \choose k}{n \choose r-k}} \right) }x^{r-1}}= \sum_{r=1}^{m+n}{m{m+n-1 \choose r-1}}x^{r-1}
    \end{aligned}
\end{equation*}
, and results in
\begin{equation}
    \begin{aligned}
        \sum_{k=0}^{r}{k{m \choose k}{n \choose r-k}} = m {m+n-1 \choose r-1}.
    \end{aligned}
    \label{eq:rightSidesAreEqual}
\end{equation}
\end{proof}

\textbf{Theorem}~\ref{1storder} is a leading part of an alternative solution to the problem in section~\ref{seq:app}.

However, going further, we realized that the \textbf{Theorem}~\ref{1storder} is a particular case of a more general form. Thus, we introduce the general form of \textit{Vandermonde's identity} as the below theorem. 

\begin{theorem}[General Form of Vandermonde's Identity] \label{generalVI} Let assume $l\leq k$, $k\leq m$ and $r-k\leq n$ then,
\begin{equation}
\sum_{k=0}^{r}{{k \brack l}{m \choose k}{n \choose r-k}} = {m \brack l} {m+n-l \choose r-l}
\label{eq:generalizedFormula}
\end{equation}
, where ${x \brack y}=\frac{x!}{(x-y)!}$ and ${x \choose y}=\frac{x!}{(x-y)!y!}$, and $x,y\in  \mathbb{N}$.
\end{theorem}

\begin{proof}
Let $f(x)= (1 + x)^{m}$. its binomial expansion is
\begin{equation*}
    f(x) = (1 + x)^{m} = \sum_{i=0}^{m}{{m \choose i}x^{i}}.
\end{equation*}

Then, utilizing induction, its $l$-th derivative is 
\begin{equation}
    f^{\left( l \right)}(x) = {m \brack l} (1+x)^{m-l} = \sum_{i=0}^{m}{{i \brack l}{m \choose i}x^{i-l}}.
    \label{eq:mthderivative} 
\end{equation}
Now, we start with ${m \brack l} \left( x+1\right)^{m+n-l}$. We have 
\begin{equation*}
    \begin{aligned}
        {m \brack l} \left( x+1 \right)^{m+n-l} & = {m \brack l} \left( x+1 \right)^{m-l} \left( x+1 \right)^{n}\\
     \end{aligned}
\end{equation*}
 using equation~\ref{eq:mthderivative} and in a similar way to the corresponding steps of \textbf{Theorem}~\ref{1storder}, we utilize the \textit{polynomial ring} and the equation is as follows.      
\begin{equation*}
    \begin{aligned} 
        & = \sum_{i=0}^{m}{{i \brack l}{m \choose i} x^{i-l}} \sum_{j=0}^{n}{{n \choose j} x^{j}}\\
        & = \sum_{r=0}^{m+n}{{ \left( \sum_{k=0}^{r}{{k \brack l}{m \choose k}{n \choose r-k}} \right) }x^{r-l}}\\
        & = 0 + \sum_{r=l}^{m+n}{{ \left( \sum_{k=0}^{r}{{k \brack l}{m \choose k}{n \choose r-k}} \right) }x^{r-l}}
    \end{aligned}
\end{equation*}
, or
\begin{equation}
    \begin{aligned} 
        {m \brack l} \left( x+1 \right)^{m+n-l} & = \sum_{r=l}^{m+n}{{ \left( \sum_{k=0}^{r}{{k \brack l}{m \choose k}{n \choose r-k}} \right) }x^{r-l}}
    \end{aligned}
    \label{eq:gen1stExpansion}
\end{equation}

On the other hand, we use binomial expansion and have

\begin{equation*}
    \left( x+1 \right)^{m+n-l} = \sum_{r=0}^{m+n-l}{{m+n-l \choose r}}x^{r}
    \label{eq:gen2ndExpansion}
\end{equation*}

changing the index variable and multiplying both sides with ${m \brack l}$

\begin{equation}
    {m \brack l} \left( x+1 \right)^{m+n-l} = \sum_{r=l}^{m+n}{{m \brack l}{m+n-l \choose r-l}}x^{r-l}
    \label{eq:gen2ndExpansion}
\end{equation}

The left-hand sides of both equations~\ref{eq:gen1stExpansion} and~equation \ref{eq:gen2ndExpansion} are the same, so the right-hand side of both equations are the same as well, or simply 
\begin{equation}
    \begin{aligned}
    \sum_{k=0}^{r}{{k \brack l}{m \choose k}{n \choose r-k}} = {m \brack l} {m+n-l \choose r-l}.
    \end{aligned} 
\end{equation}
\end{proof}

\begin{remark}
In \textbf{Theorem}~\ref{generalVI}, if
\begin{enumerate}
\item $l=0$, then it is equivalent with Vandermonde's identity, or 
\begin{equation}
    \begin{aligned}
    \sum_{k=0}^{r}{{m \choose k}{n \choose r-k}} =  {m+n \choose r};
    \end{aligned} 
\end{equation}
\item $l=1$, then it is equivalent with first order of general order of Vandermonde's identity;
\item $l>1$, then it is equivalent to higher order of general form of Vandermonde's identity.
\end{enumerate}
\end{remark}
\section{Application.}~\label{seq:app} In this section, we present the query which led us to face the first order of Vandermonde's identity general form. We prove the problem using \textbf{Theorem}~\ref{1storder}.
\subsection{Statement of the Problem.}

``An urn has $r$ red and $w$ white balls that are randomly removed one at a time. Let $R_{i}$ be the event that the $i$th ball removed is red. Find $P\left( R_{i} \right)$''~\cite{ross10}.

\textbf{Solution.} The answer is $\frac{r}{r+w}$. \textit{Also sprach Ross} ``... each of the $r+w$ balls is equally to be the $i$th ball removed''. In addition to the intuition introduced by Sheldon Ross, we would like to have a direct alternative formal proof of the answer. Thus, we tried to prove it, and here it is.
\begin{proof}

Let assume $r_{j}$ denotes choosing of $j$ red balls in $i-1$ previous steps. Then,

\begin{equation*}
        P\left(R_{i}\right) = \sum_{j=0}^{i-1}{P\left( R_{i} \cap r_{j} \right)}\\
\end{equation*} 
By applying \textit{Law of total probability} we have, 
\begin{equation*}
        = \sum_{j=0}^{i-1}{P\left( R_{i} | r_{j} \right) P\left(r_{j}\right) } 
\end{equation*}
$P\left(r_{j}\right)$, as mentioned above, is equal to choosing $j$ red balls among the $i-1$ balls that we have already picked. Thus, its value is equal to $\frac{{ r \choose j}{ w \choose i-1-j}}{{r+w \choose i-1}}$.  In addition, $P\left( R_{i} | r_{j}\right)$ denotes choosing a red ball in the $i$th step. Because, we have already picked $j$ red balls in the previous steps, its value is equal to  $\frac{{r-j \choose 1}}{{r+w-\left( i-1 \right) \choose 1}}$. Consequently, we have,

\begin{equation*}   
\begin{aligned}
       & = \sum_{j=0}^{i-1}{\left(\frac{{r-j \choose 1}}{{r+w-\left( i-1 \right) \choose 1}} \frac{{ r \choose j}{ w \choose i-1-j}}{{r+w \choose i-1}}\right)} \\
       &= \sum_{j=0}^{i-1}{\left(\frac{(r-j)}{(r+w-i+1)} \frac{{ r \choose j}{ w \choose i-1-j}}{{r+w \choose i-1}}\right)}\\
         & = \frac{1}{ \left(r+w-i+1\right) {r+w \choose i-1}} \sum_{j=0}^{i-1}{\left((r-j) { r \choose j}{ w \choose i-1-j}\right)} \\
      &  = \frac{1}{\left(r+w-i+1\right) {r+w \choose i-1}}
        \left(\sum_{j=0}^{i-1}{r{ r \choose j}{ w \choose i-1-j}} - \sum_{j=0}^{i-1}{j{ r \choose j}{ w \choose i-1-j}}\right)\\
\end{aligned}
\end{equation*}
The $\sum_{j=0}^{i-1}{j{ r \choose j}{ w \choose i-1-j}}$ term in the above formula is in accordance with \textbf{Theorem}~\ref{1storder}. Thus, we proceed with it as follows,
\begin{equation*}
\begin{aligned}
       & = \frac{1}{\left(r+w-i+1\right) {r+w \choose i-1}}
        \left(r{r+w \choose i-1} - r {r+w-1 \choose i-2 } \right)\\
       & = \frac{r}{\left(r+w-i+1\right) {r+w \choose i-1}}
        \left({r+w \choose i-1} -  {r+w-1 \choose i-2 } \right)\\
      &  = \frac{r}{\left(r+w-i+1\right) {r+w \choose i-1}}
        \left({r+w \choose i-1} - \frac{i-1}{r+w}{r+w \choose i-1 } \right)\\
     &   = \frac{r{r+w \choose i-1 }}{\left(r+w-i+1\right) {r+w \choose i-1}}
        \left( 1 - \frac{i-1}{r+w} \right) \\
      &  = \frac{r}{\left(r+w-i+1\right)}
        \left( \frac{\left(r+w-i+1\right)}{r+w} \right)\\
      &  = \frac{r}{r+w}.
    \end{aligned}
    \label{eq:ithRedBallProbability}
\end{equation*}
\end{proof}

\section{Final Thoughts} There have been many efforts to expand or extend Vandermonde's identity or Vandermonde's convolution. For example, Graham et al.~\cite{knuth94} defined some general forms of identities, including Vandermonde's identity, in chapter~5. But, that generalization is different from our point of view. They are a generalization of Vandermonde's identity, discussing the coefficients of $m$ or $n$ in the equation. The same happens in works such as~\cite{fang07, mestrovic18 }. Yaccov~\cite{yaacov17} investigated the determinant of the Vandermonde in a higher dimension. However, to our knowledge, we proposed this type of generalization of Vandermonde's identity for the first time, or at least we have given a new way of this generalization proof. In our proposal, we consider the general form of Vandermonde's identity from the point of view of Combinatorics' coefficients. It is worth mentioning that it is possible to observe our proposed generalization as a higher derivative of Vandermonde's identity.


\bibliographystyle{amsplain}

\end{document}